\documentclass[runningheads,a4paper]{llncs}
\usepackage{amsfonts}
\usepackage{enumerate}
\usepackage{xcolor}
\usepackage{tikz}
   \usetikzlibrary{calc}
\usepackage{float}
\usepackage{wrapfig}

\begin{document}

\mainmatter

\title{Edge-dominating cycles, $k$-walks and Hamilton prisms in $2K_2$-free graphs}
\titlerunning{Dominating cycles, $k$-walks and Hamilton prisms in $2K_2$-free graphs}

\author{Gao Mou\inst{1} \and Dmitrii V. Pasechnik\inst{2}}
\institute{School of Physical and Mathematical Sciences, Nanyang Technological University, Singapore,\\
\email{gaom0002@e.ntu.edu.sg} \and
Department of Computer Science, The University of Oxford, UK,\\
\email{dimpase@cs.ox.ac.uk}}

\maketitle

\begin{flushright}
    \textit{To the memory of Sergei Duzhin.}
\end{flushright}

\begin{abstract}
We show that an edge-dominating cycle in a $2K_2$-free graph can be found
in polynomial time; this implies that every $\frac{1}{k-1}$-tough
$2K_2$-free graph admits a $k$-walk, and it can be found in polynomial time.
For this class of graphs, this proves a
long-standing conjecture due to Jackson and Wormald (1990).
Furthermore, we prove that for any $\epsilon>0$ every $(1+\epsilon)$-tough $2K_2$-free graph is prism-Hamiltonian
and give an effective construction of a Hamiltonian cycle in the corresponding prism,
along with few other similar results.
\end{abstract}

\section{Introduction}
A graph $G$ is called $\beta$-{\em tough}, for a real $\beta>0$, if for any $p\geq 2$ it
cannot be split into $p$ components by removing less than $p\beta$ vertices.
This concept, a measure of graph connectivity and ``resilience'' under vertex subsets removal,
was introduced in 1973 by Chv\'{a}tal \cite{chvatal1973tough},
while studying   Hamiltonicity of graphs. For a survey of results on graph toughness till 2006
see \cite{MR2221006}.

In general, toughness of a graph is NP-hard to compute
\cite{MR1074858}. Considerable work went into investigating this computational problem for
various classes of graphs. In particular, recently, Broersma, Patel and Pyatkin proved
\cite{broersma2014toughness} that toughness of a $2K_2$-free graph, i.e. a graph that
does not contain an induced copy of the disjoint union of two edges,
can be found in polynomial time.

Note that $2K_2$-free graphs are an interesting
class from algorithmic complexity point of view; most classical algorithmic problems for them are
hard, with a notable exception of the maximum weighted independent set problem
\cite{balasyu1989}, \cite[Graphclass: $2K_2$-free]{isgci}. In particular
Hamiltonian cycle problem is NP-complete already for a subclass of $2K_2$-free graphs,
the {\em split} graphs---graphs for which the set of vertices can be partitioned into
a clique and an independent set \cite[Exercise 6.2]{Golum}.
Due to the latter,  for $2K_2$-free graphs it makes sense to study computational complexity
of concepts which are generalisations of the Hamiltonian cycle
problem, such as $k$-walk.

Let $p\times G$ denote the multigraph obtained from $G$ by taking each edge $p$ times.
A $k$-{\em walk} is a spanning subgraph $W$ of $2k\times G$ such that each vertex of $W$
has even degree at most $2k$. 
In particular a graph has a $1$-walk if and only if it is $K_2$ (i.e. one edge) or Hamiltonian.
{For a survey of results on walks in graphs till 2005 see \cite{kouider2005connected}.}
In 1990 Jackson and Wormald conjectured \cite{jackson1990k} that for any integer $k\ge2$ a
$\frac{1}{k-1}$-tough graph $G$ admits a $k$-walk.

In this paper, we prove that Jackson and Wormald's Conjecture is true under the
assumption that $G$ is  $2K_2$-free.

\begin{theorem}\label{thm2}
For any integer $k\ge2$, every $\frac{1}{k-1}$-tough $2K_2$-free graph $G$
admits a $k$-walk.
Moreover, the latter can be found in time polynomial in $|V(G)|$.
\end{theorem}

If for $k\geq 2$ we let the toughness value $\frac{1}{k-1}$ increase to
$\frac{1}{k-2}$ then
one does not need $2K_2$-freeness. Indeed, it is shown in
\cite{jackson1990k} that
every $\frac{1}{k-2}$-tough graph has a $k$-walk.

Clearly, if $G$ is
Hamiltonian, then $G$ is 1-tough.  More generally,
if $G$ has a $k$-walk, then $G$ is $\frac{1}{k}$-tough \cite{jackson1990k}.
However, the converse is not true already for $k=1$ (there even exist $2$-tough graphs which are
not Hamiltonian, cf. \cite{bauer2000not}).

This more or less summarises the situation with $t$-tough graphs, $t\leq 1$.
On the $t>1$ side
a famous conjecture of Chv\'{a}tal \cite{chvatal1973tough} claims
that there exists a constant $\beta$ such that every
$\beta$-tough graph is Hamiltonian.
Towards this,
Ellingham and Zha \cite{ellingham2000toughness} proved that
every 4-tough graph has a 2-walk (cf. Theorem~\ref{thm1} below).

It was recently shown \cite{broersma2014toughness} that
every 25-tough 2$K_2$-free graph on at least three vertices is Hamiltonian.
Our Theorem~\ref{thm2} was inspired by this result.
However, our approach is technically quite different.

\medskip

Our next result concerns a structure that is half-way between 1- and 2-walks.
The {\em prism} over a graph $G$ is the Cartesian product
\begin{wrapfigure}[6]{r}{.3\textwidth}
\centering
\setlength{\unitlength}{.5mm}
\begin{picture}(50,40)
\put(10,10){\circle*{1.5}}
\put(20,10){\circle*{1.5}}
\put(40,10){\circle*{1.5}}
\put(50,10){\circle*{1.5}}
\put(30,20){\circle*{1.5}}
\put(30,30){\circle*{1.5}}
\put(10,10){\line(1,0){40}}
\put(20,10){\line(1,1){10}}
\put(30,20){\line(1,-1){10}}
\put(30,20){\line(0,1){10}}
\end{picture}
\label{fignoprism}
\end{wrapfigure}
$G\square K_2$ of $G$ with the complete graph $K_2$.
$G$ is called
{\em prism-Hamiltonian} if $G\square K_2$ is Hamiltonian.
If $G$ is Hamiltonian, then $G\square K_2$ is also Hamiltonian, but the converse does not hold in general,
cf. \cite{kaiser2007hamilton}.
As well, this property is stronger than having a $2$-walk: cf. figure on the right, where
we have a
$2K_2$-free graph with  a $2$-walk, but without Hamiltonian prism.

\begin{theorem}\label{thm1}
Every $(1+\epsilon)$-tough $2K_2$-free graph $G$ is prism-Hamiltonian, for any $\epsilon>0$.
Moreover, a Hamiltonian cycle in the prism over $G$ can be found in time polynomial in $|V(G)|$.
\end{theorem}

It is worth mentioning that the toughness constant in Theorem~\ref{thm1} is much better
than $\frac{3}{2}$, the lower bound on toughness of a $2K_2$-free graph needed
for its Hamiltonicity, see \cite[Sect.~4]{broersma2014toughness}.

To prove Theorem~\ref{thm2} and Theorem \ref{thm1}, we first prove
a result on edge-dominating subgraphs (a subgraph $S$ of $G$ is called {\em edge-dominating}
if each edge of $G$ contains at least one vertex from $V(S)$).
\begin{theorem}\label{addgen1}
Let $G$ be a $2K_2$-free graph. Then
\begin{enumerate}
\item $G$ admits an edge-dominating cycle (or an edge, or a vertex) $C$;
\item if $G$ contains a triangle, then $G$ admits
an edge-dominating cycle $C$, with three successive vertices on $C$ forming a triangle in $G$.
\end{enumerate}
Moreover, $C$ can be found in time polynomial in $|V(G)|$.
\end{theorem}

In fact, in 1983 Veldman \cite{veldman83} has proved the existence of
edge-dominating cycles for $2K_2$-free graphs. However, his proof is based on
contraposition, so it neither tells how to find $C$ in (1), nor
allows to restrict $C$ as in (2).

In the remainder of the paper we provide the proofs, and then discuss related
open questions.

\section{Proof of Theorem \ref{addgen1}}

\subsection{The proof of the first part of Theorem \ref{addgen1}}
If $G$ is a tree, then, as it is $2K_2$-free, it must either have an edge-dominating
vertex, or an edge-dominating edge.
Indeed, in this case, if any two edges intersect, then, as $G$ has no cycles,
they all must intersect in a vertex, which will be dominating. Now, assume that there are
two non-intersecting edges, say $xy$ and $uv$. As they cannot form a $2K_2$, they are on
a 3-path, $xyuv$, without loss of generality. Now we claim that $yu$ is a dominating edge.
Indeed, suppose to the contrary that an edge $ab$ does not intersect $yu$. Then, as they
cannot form $2K_2$, $abyu$ is a 3-path, without loss of generality. As $G$ is a tree, $by$
is the only edge connecting vertices of $ab$ and $yu$. Thus either
$ab$ and $uv$ form a $2K_2$, or $b$ lies in a cycle; in both cases this is a contradiction,
proving that $yu$ is a dominating edge.

Otherwise, $G$ has a cycle, say $C=x_1x_2\cdots x_kx_1$, where $k\ge3$.
If $C$ is edge-dominating, then we are done.
Otherwise
there must be an edge $v_1v_2$ (assume there are $t>0$ such edges),
with neither $v_1$ nor $v_2$ on $C$.
Since $G$ is $2K_2$-free, $v_1$ and $v_2$ have at least two distinct neighbours on $C$.
Let $x_1v_1\in E(G)$ without loss of generality;
\begin{enumerate}
\item if $x_2v_1\in E(G)$, then $C'=x_1v_1x_2x_3\cdots x_kx_1$ is a longer cycle;
\item if $x_2v_2\in E(G)$, then $x_1v_1v_2x_2x_3\cdots x_kx_1$ is a longer cycle;
\item if $x_2v_1,x_2v_2\not\in E(G)$, then applying $2K_2$-freeness to $v_1v_2$ and $x_2x_3$,
we get either $x_3v_1\in E(G)$ or $x_3v_2\in E(G)$.
\begin{enumerate}
\item If $x_3v_2\in E(G)$, then $C'=x_1v_1v_2x_3\cdots x_kx_1$ is a longer cycle;
\item if $x_3v_2\not\in E(G)$, then $x_3v_1\in E(G)$.
\begin{enumerate}
\item if $x_2$ is adjacent to no vertex outside $C$, then use $C'=x_1v_1x_3\cdots x_kx_1$ instead of $C$.
We know that $C$ and $C'$ have the same length, but $C'$ dominates all the edges that
are dominated by $C$, and $C'$ also dominates $v_1v_2$, which is not dominated by $C$. So replacing $C$ by $C'$
decreases $t$.
\item Otherwise $x_2$ is adjacent to a vertex outside $C$, say $z$.
As $x_2$ is not adjacent to $v_1$ or $v_2$, we have $z$ adjacent to either $v_1$ or $v_2$. If $zv_1\in E(G)$, then $C'=x_1v_1zx_2x_3\cdots x_kx_1$ is a longer cycle.
Otherwise $C'=x_1v_1v_2zx_2x_3\cdots x_kx_1$ is a longer cycle.
\end{enumerate}
\end{enumerate}
\end{enumerate}
Repeat the process above. At each iteration either $|V(C)|$ increases, or $t$ decreases.
Thus the process will stop, with $t=0$, in at most $|E(G)|^2$ steps.
\qed

\subsection{The proof of the second part of Theorem  \ref{addgen1}}
The algorithmic procedure for the second part is almost the same, requiring only a minor
modification described below.

Let $G$  contain a triangle $w_1w_2w_3$.  If $w_1w_2w_3$ is edge-dominating,
then there is nothing to prove. Otherwise, there is $u_1u_2\in
E(G)$, with neither $u_1$ nor $u_2$  on $w_1w_2w_3$.  Then, by the
$2K_2$-freeness, we either can connect $w_1w_2w_3$ and $u_1u_2$ together, to get a
5-cycle $C$, with $w_{\pi(1)}$, $w_{\pi(2)}$ and $w_{\pi(3)}$ successive on $C$
for some permutation $\pi$ of $\{1,2,3\}$, or else (without loss in generality)
$u_1$ is adjacent to $w_1$ and $w_2$. In the latter case set $C=u_1w_1w_3w_2$.

If $C$ is edge-dominating, then we are done. Otherwise, we proceed by induction on
$k:=|V(C)|$.
Suppose $k\ge 4$, and there are three successive vertices on $C$, namely $X'$, $X$ and $X''$ forming a triangle in $G$.
Let $v_1v_2\in E(G)$ such that neither $v_1$ nor $v_2$ is on $C$.
We claim that then we can find a cycle $C'$ such that $C'$ dominates more edges than $C$ (perhaps all),
and $X'$, $X$ and $X''$ are also successive on $C'$.

Now we dispense with the case $k=4$. By construction, there is an edge $u_1u_2$
with $u_2$ not in $C=u_1w_1w_3w_2$. If $u_2$ is joined to a vertex $w\neq u_1$ on $C$,
then we use this extra edge to turn $C$ into a 5-cycle through $u_1u_2$
and $w_j$ ($j=1,2,3$). Indeed, as $w_j$ ($j=1,2,3$) form a triangle, they will give,
in some order, three successive vertices $X'XX''$ forming a triangle, as required.
Namely, if $w=w_1$ or $w=w_2$, we will have $X'XX''=w_1w_3w_2$, whereas for $w=w_3$
we will have $X'XX''=w_3w_2w_1$.
Thus, we are left with the case where the only edge joining $u_2$ and $C$ is $u_1u_2$.

First, consider the case $v_2=u_2$. Then $v_1$ must be adjacent to the edge $w_1w_2$,
otherwise $v_1u_2$ and $w_1w_2$ form a $2K_2$. Thus we obtain a 5-cycle $\Omega(u_2)$
through $u_1u_2$, $v_1$, $w_1$, and $w_2$, with $\{X,X',X''\}=\{u_1,w_1,w_2\}$.

It remains to consider the case of $v_i\neq u_2$ ($i=1,2$).
As $u_1u_2$ and $v_1v_2$ cannot form a $2K_2$, and as we already considered
the case of $u_2$ adjacent to a vertex not on $C$, we may assume that
$v_2$ is adjacent to $u_1$. This gives us the already considered configuration,
with $u_2$ replaced by $v_2$, from which we obtain a 5-cycle $\Omega(v_2)$.

\medskip

From now on we can assume $k\geq 5$.
By $2K_2$-freeness, $v_1$ and $v_2$ are adjacent to at least two of $\{X,X',X''\}$,
and thus to at least one of $\{X',X''\}$.  Suppose,
without loss of generality, that $v_1X'\in E(G)$; label the vertices in $C$ in the following way:
$X'$ is labeled by $x_1$; the neighbour of $x_1$ on $C$ distinct from $X$ is labeled by $x_2$;
the other vertices on $C$ are labeled successively, see Figure \ref{labelcycle}.

\begin{figure}[h]
\begin{center}
\begin{tikzpicture}

\filldraw[black](0,0)node[left]{$X'=x_1$} circle(2pt);
\filldraw[black](1.8,0)node[above right]{ $x_2$} circle(2pt);
\filldraw[black](3.5,-0.5)node[above]{ $x_3$} circle(2pt);
\draw[very thick](0,0)--(1.8,0);
\draw[very thick](3.5,-0.5)--(1.8,0);
\draw[very thick](3.5,-0.5)--(4,-1);
\draw[dashed](4,-1)--(4.5,-1.5);

\filldraw[black](-1.2,-1)node[left]{$X=x_k$} circle(2pt);
\draw[very thick](0,0)--(-1.2,-1);

\filldraw[black](-1.2,-2)node[left]{$X''=x_{k-1}$} circle(2pt);
\draw[very thick](-1.2,-2)--(-1.2,-1);
\draw[very thick](-1.2,-2)--(-1,-2.2);
\draw[dashed](-1,-2.2)--(-0.6,-2.6);

\filldraw[black](0,1)node[below right]{ $v_1$} circle(2pt);
\draw[very thick](0,1)--(1.5,1);
\filldraw[black](1.5,1)node[below right]{ $v_2$} circle(2pt);
\draw[very thick](0,1)--(0,0);
\draw[thin](0,0)--(-1.2,-2);

\end{tikzpicture}
\end{center}
\caption{The cycle $C$ of length $k$}\label{labelcycle}
\end{figure}
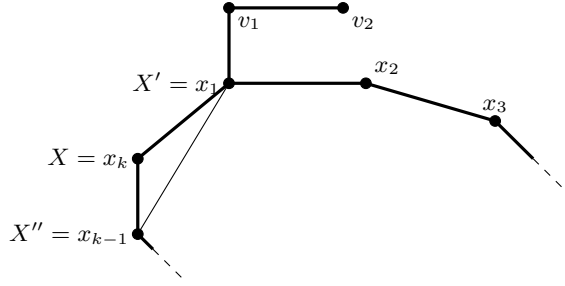

Note that for $k\geq 5$ the operation used
in the proof of the first part of Theorem \ref{addgen1}
of replacing  $C$ by $C'$ (enlarging $|V(C)|$ or reducing $t$)
does not touch the edges $x_{k-1}x_k$ and $x_kx_1$. Thus the triangle-forming
vertices $X''$, $X$ and $X'$ are always successive on $C$ in our process.
Then they are on the edge-dominating cycle we obtain there.  \qed

\section{Proof of Theorem \ref{thm2}}

Combining Theorem \ref{addgen1} and the following
Lemma \ref{addtec}, we obtain Theorem \ref{thm2}.

\begin{lemma}\label{addtec}
Let $k\geq 2$.
If $G$ has an edge-dominating cycle $C$ (or an edge, or a vertex)
and if $G$ is $\frac{1}{k-1}$-tough then $G$ admits a $k$-walk.
Moreover, the latter can be found in time polynomial in $|V(G)|$.
\end{lemma}
\begin{proof}
The induced subgraph $D=G-C$ is an independent set. Define an auxiliary bipartite
graph $\Gamma$, with one part $V(D)$ and the other consisting of $k-1$ copies of $V(C)$,
so that each edge $dc\in V(D)\times V(C)$ corresponds to $k-1$ edges
$dc_1,\dots dc_{k-1}$ of $\Gamma$. Let $\Phi$ denote the $(k-1)$-to-$1$ map
$\Phi: E(\Gamma)\to E(G)$ sending each $dc_j$, for $1\leq j\leq k-1$, to $dc=\Phi(dc_j)$.

For any $D'\subset D$, by $\frac{1}{k-1}$-toughness, $D'$ has at least
$\lceil\frac{|D'|}{k-1}\rceil$ neighbours in $C$. Thus $D'$ has at
least $|D'|$ neighbours in $\Gamma$.  By Hall's Theorem \cite[Theorem 16.4]{bomu08} applied
to $\Gamma$, it has a matching $M$ saturating $D$; i.e.  each $v\in V(D)$ is
incident to an $e\in M$, and each $v$ in the other part of $\Gamma$ is
incident to at most one $e\in M$.

Hence for each $e\in\Phi(M)\subset E(G)$ we have $e\in D\times C$.
Moreover, each $v\in V(D)$ is incident to
exactly one $e\in \Phi(M)$, while each $v\in V(C)$ is incident to at most
$k-1$ edges in $\Phi(M)$. Then the (doubled) edges in $\Phi(M)$ and the edges
in the edge-dominating cycle (respectively, the doubled edge in the case of existence of a dominating edge) $C$ form a
$k$-walk in $G$.

To show the last claim, it suffices to notice that $M$ (a maximum matching
in a bipartite graph) can be found in time polynomial
in $|V(G)|$. \qed
\end{proof}

\section{Proof of Theorem \ref{thm1}}
The following lemma is the key technique in the proof of Theorem \ref{thm1}.
\begin{lemma}\label{keylem}
Let $G$ be $(1+\epsilon)$-tough, for some $\epsilon>0$.
\begin{enumerate}
\item If $G$ contains an edge-dominating cycle $C$ with even number of vertices, then the prism over $G$ is Hamiltonian.
\item If $G$ contains an edge-dominating cycle $C=v_1v_2\cdots v_{2p+1}v_1$ of odd length, and there are three vertices $v_1$, $v_{2q}$ and $v_{2q+1}$, for some $1\le q\le p$, inducing a triangle in $G$, then the prism over $G$ is Hamiltonian.
\end{enumerate}
\end{lemma}

\begin{proof}
For the first part (see Figure~\ref{fig:evenprism}),
denote $C=v_1v_2\cdots v_{2p}v_1$. The set
$D=V(G)-V(C)$ of vertices outside $C$
is an independent set. By Hall's Theorem and 1-toughness, there is a matching $M$ between
$D$ and $C$. That means that for any vertex $u_j$ in $D$, there is a vertex $v_{i_j}$ on $C$ adjacent to $u_j$ in $M$.

Obviously, we have a Hamiltonian cycle in $\bar{C}$, the prism over $C$, namely $$v_1v'_1v'_2v_2\cdots v_{2p-1}v'_{2p-1}v'_{2p}v_{2p}v_1.$$ Now, we change every  $v_{i_j}v'_{i_j}$ (or $v'_{i_j}v_{i_j}$) into $v_{i_j}u_ju'_jv'_{i_j}$ (or $v'_{i_j}u'_ju_jv_{i_j}$) to get a Hamilton cycle in $\bar{G}$.

\begin{figure}[h]
\begin{center}
\begin{tikzpicture}[scale=2]
\draw(-2,0)node[below left]{$v_{2p}$}--(-1,0);\draw(-0.5,0)--(0.5,0);\draw(1,0)--(2,0);
\draw(-2,0.6)node[below left]{$v'_{2p}$}--(-1,0.6);\draw(-0.5,0.6)--(0.5,0.6);\draw(1,0.6)--(2,0.6);
\filldraw[black](0,0)node[below left]{$v_{i_j}$} circle(1pt);
\filldraw[black](0,0.6)node[below left]{$v'_{i_j}$} circle(1pt);
\filldraw[black](-2,0.6) circle(1pt);
\filldraw[black](2,0.6)node[below left]{$v'_1$} circle(1pt);
\filldraw[black](-1.5,0.6)node[below right]{$v'_{2p-1}$} circle(1pt);
\filldraw[black](1.55,0.6)node[below left]{$v'_2$} circle(1pt);
\filldraw[black](2,0)node[below left]{$v_1$} circle(1pt);
\filldraw[black](-2,0) circle(1pt);
\filldraw[black](1.55,0)node[below left]{$v_2$} circle(1pt);
\filldraw[black](-1.5,0)node[below right]{$v_{2p-1}$} circle(1pt);
\draw[dashed](-1,0)--(-0.5,0);
\draw[dashed](1,0)--(0.5,0);
\draw[dashed](-1,0.6)--(-0.5,0.6);
\draw[dashed](1,0.6)--(0.5,0.6);
\filldraw[black](1.1,0)node[below left]{$v_3$} circle(1pt);
\filldraw[black](1.1,0.6)node[below left]{$v'_3$} circle(1pt);
\filldraw[black](0.2,0.4)node[below right]{$u_j$} circle(1pt);
\filldraw[black](0.2,1)node[below right]{$u'_j$} circle(1pt);
\draw[very thick](2,0)--(2,0.6)--(1.55,0.6)--(1.55,0)--(1.1,0)--(1.1,0.6);
\draw[very thick](0,0)--(0.2,0.4)--(0.2,1)--(0,0.6);
\draw[dashed, very thick](0,0)--(0,0.6);
\draw[very thick](-1.5,0)--(-1.5,0.6)--(-2,0.6)--(-2,0);
\draw[very thick](-2,0)..controls(0,-0.5)..(2,0);
\draw(-2,0.6)..controls(0,0.9)..(2,0.6);
\node at(0.2,-0.5){$G$};\node at (-0.4,0.9){$G'$};
\end{tikzpicture}
\end{center}
\caption{$G$ has an edge-dominating cycle of even length\label{fig:evenprism}}
\end{figure}
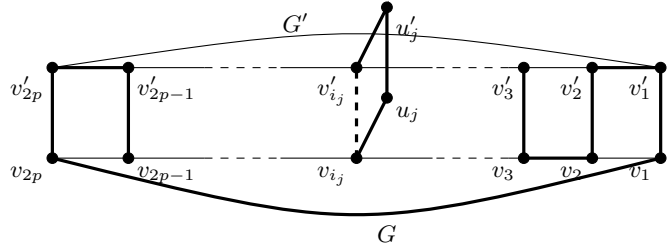

For the second part, denote $C=v_1v_2\cdots v_{2p+1}v_1$ (see Figure \ref{fig2}).
The set $D=V(G)-V(C)$ of vertices outside $C$
is an independent set. By Hall's Theorem, and $(1+\epsilon)$-toughness,
there is a matching $M$ between $D$ and $C-\{v_1\}$.  This means that for any vertex $u_j$ in $D$,
there is a vertex $v_{i_j}$ in $C-\{v_1\}$ adjacent to $u_j$ in $M$.

Clearly, we have a Hamiltonian cycle in $\bar{C}$, namely $$v_1v_2v'_2v'_3v_3\cdots v_{2q-1}v_{2q}v'_{2q}v'_1v'_{2q+1}v_{2q+1}\cdots v_{2p+1}v_1.$$
Now, we change every  $v_{i_j}v'_{i_j}$ (or $v'_{i_j}v_{i_j}$) into $v_{i_j}u_ju'_jv'_{i_j}$ (or $v'_{i_j}u'_ju_jv_{i_j}$) to get a Hamilton cycle in $\bar{G}$. \qed
\end{proof}

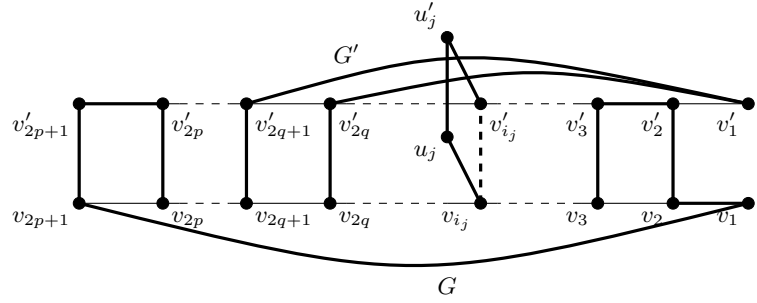
\begin{figure}[h]
\begin{center}
\begin{tikzpicture}[scale=2.2]
\draw(-2,0)node[below left]{$v_{2p+1}$}--(-1.4,0);\draw(1,0)--(2,0);
\draw(-2,0.6)node[below left]{$v'_{2p+1}$}--(-1.4,0.6);\draw(1,0.6)--(2,0.6);
\filldraw[black](0.4,0)node[below left]{$v_{i_j}$} circle(1pt);
\filldraw[black](0.4,0.6)node[below right]{$v'_{i_j}$} circle(1pt);
\filldraw[black](-2,0.6) circle(1pt);
\filldraw[black](2,0.6)node[below left]{$v'_1$} circle(1pt);
\filldraw[black](-1.5,0.6)node[below right]{$v'_{2p}$} circle(1pt);
\filldraw[black](1.55,0.6)node[below left]{$v'_2$} circle(1pt);
\filldraw[black](2,0)node[below left]{$v_1$} circle(1pt);
\filldraw[black](-2,0) circle(1pt);
\filldraw[black](1.55,0)node[below left]{$v_2$} circle(1pt);
\filldraw[black](-1.5,0)node[below right]{$v_{2p}$} circle(1pt);
\draw[dashed](-1.4,0)--(-1.1,0);\draw(-1.1,0)--(-0.4,0);\filldraw[black](-1,0)node[below right]{$v_{2q+1}$} circle(1pt);
\draw[dashed](1,0)--(0.5,0);
\draw[dashed](-1.4,0.6)--(-1.1,0.6);\draw(-1.1,0.6)--(-0.4,0.6);\filldraw[black](-1,0.6)node[below right]{$v'_{2q+1}$} circle(1pt);
\draw[dashed](1,0.6)--(0.5,0.6);
\filldraw[black](1.1,0)node[below left]{$v_3$} circle(1pt);
\filldraw[black](1.1,0.6)node[below left]{$v'_3$} circle(1pt);
\filldraw[black](0.2,0.4)node[below left]{$u_j$} circle(1pt);
\filldraw[black](0.2,1)node[above left]{$u'_j$} circle(1pt);
\draw[very thick](2,0)--(1.55,0)--(1.55,0.6)--(1.1,0.6)--(1.1,0);
\draw[very thick](0.4,0)--(0.2,0.4)--(0.2,1)--(0.4,0.6);
\draw[dashed, very thick](0.4,0)--(0.4,0.6);
\draw[very thick](-1.5,0)--(-1.5,0.6)--(-2,0.6)--(-2,0);
\draw[very thick](-2,0)..controls(0,-0.5)..(2,0);

\node at(0.2,-0.5){$G$};\node at (-0.4,0.9){$G'$};
\draw[dashed](-0.4,0)--(0.2,0);\draw(0.2,0)--(0.5,0);
\draw[dashed](-0.4,0.6)--(0.2,0.6);\draw(0.2,0.6)--(0.5,0.6);
\filldraw[black](-0.5,0)node[below right]{$v_{2q}$} circle(1pt);
\filldraw[black](-0.5,0.6)node[below right]{$v'_{2q}$} circle(1pt);
\draw[very thick](-0.5,0.6)..controls(0.7,0.85)..(2,0.6);
\draw[very thick](-1,0.6)..controls(0.3,0.97)..(2,0.6);
\draw[very thick](-0.5,0.6)--(-0.5,0);
\draw[very thick](-1,0.6)--(-1,0);
\end{tikzpicture}
\end{center}
\caption{$G$ has an edge-dominating cycle of odd length}\label{fig2}
\end{figure}

Now we complete the proof of Theorem \ref{thm1}.
Suppose $G$ is a triangle-free $2K_2$-free graph. By \cite[Theorem 4]{broersma2014toughness},
if $|V(G)|\geq 3$ then
$G$ is Hamiltonian,\footnote{
We only need to rely on  \cite[Theorem 4]{broersma2014toughness}
for the case of $G$ not having a dominating cycle $C$ of even length, for
otherwise we have case 1 of Lemma \ref{keylem} at our disposal; $|C|$ can be odd
only in the case of $G$ not bipartite; such a $G$ is, by \cite[Lemma 2]{broersma2014toughness},
of $C_5^*$-type, i.e. it has a rather special structure.}
and so prism-Hamiltonian.
A polynomial-time algorithm to construct a Hamiltonian cycle in $G$ is
implicit in the proof of  \cite[Theorem 4]{broersma2014toughness}, and
building up a Hamiltonian cycle in the prism over $G$ from a Hamiltonian
cycle $C=v_1v_2\cdots v_{m}v_1$ in $G$ is trivially (and in time polynomial in $|V(G)|$)
done by concatenating the path $v_1\dots v_{m-1}v_m$, i.e. $C$ without the last edge, with the
path $v'_m v'_{m-1}\cdots v'_1$.

If $|V(G)|=2$, i.e. $G$ is a single edge, and obviously prism-Hamiltonian.
Finally, if $G$ is not triangle-free then we are done by
Theorem \ref{addgen1} (2) and Lemma \ref{keylem}, and noting that the
corresponding construction can be done in  time polynomial in $|V(G)|$.
\qed

\section{Concluding remarks}

Lemma \ref{addtec} can be used to prove existence of 2-walks in
classes of graphs wider than $2K_2$-free.
For instance, it is immediate from \cite[Corollary 3.2]{veldman83} that each 2-connected
$3K_2$-free graph admits  an edge-dominating cycle. From the latter and Lemma \ref{addtec},
it is easy to obtain the following.
\begin{theorem}\label{2w3k2f}
Let $G$ be a 1-tough $3K_2$-free graph. Then $G$ admits a 2-walk. \qed
\end{theorem}

It would be interesting to find out whether Theorem~\ref{2w3k2f} and similar results
of this type can be made
effective. Towards this end, we would like to propose the following
\begin{conjecture}
Let $\ell\geq 2$ be a fixed constant.
Then for the $\ell-1$-connected $\ell K_2$-free
graphs there is a polynomial time algorithm finding an edge-dominating cycle.
\end{conjecture}

Of independent interest would be finding out whether more general results from \cite{veldman83},
in particular Theorem~\ref{thm:veld}, can be made algorithmic.
\begin{theorem}\label{thm:veld}
{\rm \cite[Theorem~3]{veldman83}.} Let $G$ be an $\ell-1$-connected graph such that
for every induced $\ell K_2$-subgraph $H$ of $G$ one has the
sum of degrees of vertices in $H$ at least $\frac{(\ell-1)(|V(G)-\ell +1)}{2}$.
Then $G$ has  an edge-dominating cycle. \qed
\end{theorem}

\subsection*{Acknowledgements.}
The authors thank Nick Gravin, Edith Elkind, and the anonymous referee
for helpful comments on versions of this text.
Research supported by Singapore MOE Tier 2 Grant MOE2011-T2-1-090 (ARC 19/11)
and by the EU Horizon 2020 research and innovation programme, grant agreement
OpenDreamKit No 676541.

\bibliography{reftough}
\bibliographystyle{splncs03}
\end{document}